\documentclass[11pt]{amsart}

\usepackage{amsthm}                   
\usepackage{amsmath}                  
\usepackage{amsfonts}                 
\usepackage{amssymb}                  
\usepackage{mathrsfs}                 
\usepackage{bbm}                      
\usepackage{mathtools}								
\usepackage[margin=1in]{geometry}   	

\theoremstyle{definition}
\newtheorem{defi}{Definition}
\newtheorem{ex}{Example}

\theoremstyle{plain}
\newtheorem{thm}[defi]{Theorem}
\newtheorem{lemma}[defi]{Lemma}
\newtheorem{cor}[defi]{Corollary}
\newtheorem{prop}[defi]{Proposition}
\newtheorem{fact}{Fact}

\numberwithin{equation}{section}

\newcommand{\R}{\ensuremath{\mathbbm{R}}}     
\newcommand{\N}{\ensuremath{\mathbbm{N}}}     
\newcommand{\Z}{\ensuremath{\mathbbm{Z}}}     
\newcommand{\Q}{\ensuremath{\mathbbm{Q}}}     
\renewcommand{\H}{\ensuremath{\mathbbm{H}}}   
\newcommand{\J}{\ensuremath{\mathbbm{J}}}     
\renewcommand{\L}{\ensuremath{\mathbbm{L}}}   

\def\lt{\left}
\def\rt{\right}

\newcommand{\Rd}{\ensuremath{\R^d}}                       
\renewcommand{\Re}[1]{\ensuremath{\mathrm{Re}\lt(#1\rt)}} 
\renewcommand{\Im}[1]{\ensuremath{\mathrm{Im}\lt(#1\rt)}} 
\newcommand{\set}[1]{\left\{#1\right\}}                   
\newcommand{\abs}[1]{\left|#1\right|}                     
\newcommand{\abss}[1]{{\left|#1\right|}^2}                
\newcommand{\inn}[2]{\lt\langle #1,#2\rt\rangle}          
\newcommand{\nrq}{{|q|}^2}                                
\newcommand{\nrp}{{|p|}^2}                                
\DeclareMathOperator{\den}{den}                           
\DeclareMathOperator{\lcm}{lcm}                           
\DeclareMathOperator{\OC}{OC} 
\DeclareMathOperator{\SOC}{SOC}
\DeclareMathOperator{\SO}{SO}
\DeclareMathOperator{\OG}{O} 
\newcommand{\e}{\ensuremath{\mathbf{e}}}                  
\newcommand{\ii}{\ensuremath{\mathbf{i}}}                 
\renewcommand{\j}{\ensuremath{\mathbf{j}}}                
\renewcommand{\k}{\ensuremath{\mathbf{k}}}                

\makeatletter
\def\imod#1{\allowbreak\mkern10mu({\operator@font mod}\,\,#1)} 
\def\@setcopyright{}                                           
\def\serieslogo@{}
\makeatother

\newcommand{\G}{\ensuremath{\Lambda}}       
\renewcommand{\S}{\ensuremath{\Sigma}}      

\begin{document}
	\author[M.J.C.~Loquias]{Manuel Joseph C.~Loquias}
	\address[M.J.C.~Loquias]{Institute of Mathematics, University of the Philippines Diliman, 1101 Quezon City, Philippines, and Chair of Mathematics and Statistics, University of
	Leoben, Franz-Josef-Strasse 18, A-8700 Leoben, Austria}
	\email{mjcloquias@math.upd.edu.ph}
	
	\author[P.~Zeiner]{Peter Zeiner}
	\address[P.~Zeiner]{Fakult\"at f\"ur Mathematik, Universit\"at Bielefeld, Postfach 100131, 33501 Germany}
	\email{pzeiner@math.uni-bielefeld.de}

	\title[Coincidence indices of sublattices and Coincidences of Colorings]{Coincidence indices of sublattices and\\ coincidences of colorings}

	\begin{abstract} 
		Even though a lattice and its sublattices have the same group of coincidence isometries, the coincidence index of a coincidence isometry with respect
		to a lattice $\G_1$ and to a sublattice $\G_2$ may differ.  Here, we examine the coloring of $\G_1$ induced by $\G_2$ to identify how the 
		coincidence indices with respect to $\G_1$ and to $\G_2$ are related.  This leads to a generalization of the notion of color symmetries of lattices to what we call color
		coincidences of lattices.  Examples involving the cubic and hypercubic lattices are given to illustrate these ideas.
	\end{abstract}

	\subjclass[2010]{Primary 52C07; Secondary 11H06, 82D25, 52C23}

	\keywords{coincidence site lattice, coincidence index, color coincidences, color symmetry}

	\date{\today}

	\maketitle
	
	\section{Introduction}   
		Coincidence site lattices (CSLs) are an important tool to describe grain boundaries in crystals~\cite{F11,R66,B70} and quasicrystals~\cite{W93,WL95,OWL98,WL01}.
		In particular one is interested in the so-called coincidence index, 
		which describes how much larger the unit cell of the CSL is compared to the unit cell of the parent lattice.
		It is a well-known fact that a lattice and all its sublattices share the same group of coincidence isometries, 
		more generally - two commensurate lattices have the same coincidence isometries. 
		However, the coincidence indices of a coincidence isometry with respect to the lattice and its sublattices usually differ. 
		Nevertheless, there is a relationship between the coincidence indices for a parent lattice and its sublattice. 
		The aim of this paper is to give an explicit expression of this relationship by means of colorings of lattices.

		Grimmer, Bollmann, and Warrington proved that the coincidence indices for a given coincidence isometry
		are the same for all three cubic lattices by using a shelling argument~\cite{GBW74}.
		The main idea behind their proof can be rephrased in the terminology of colorings. 
		Consider a body-centered cubic lattice and its maximal primitive cubic sublattice of index $2$.
		Suppose we assign the color black to the lattice points of the primitive cubic lattice 
		and the color white to the other lattice points. 
		Then a rotation about any lattice point can map black points only onto black points and never onto white points, 
		and vice versa - this follows from the fact that black and white points lie on different shells.

		The aim of this paper is to generalize this idea. 
		First, we color the lattice points of a sublattice using the same color, and assign suitable colors to the remaining lattice points. 
		By analyzing how these colors are interchanged by the coincidence isometries, 
		we are able to get an explicit formula for the coincidence index of the sublattice 
		in terms of the coincidence index of its parent lattice and some properties of the coloring of the lattice.

		This establishes a link between CSLs and colorings of a lattice.
		Actually, connections between colorings and CSLs have already been observed earlier in~\cite{S80,PBR96,B97b,L08,LF08}.
		This motivates us to analyze the coincidences of lattice colorings even further.
		Indeed, there is a long tradition on analyzing colorings of various structures,
		starting from the context of magnetic structures~\cite{SB64, Litvin2013} that led to
		an extensive analysis of various symmetrically colored symmetrical structures, 
		such as colored lattices~\cite{H78a,S80,S84,DLPF07} and colored tilings~\cite{C86,S88}. 
		In fact, studies have been done not only on colored periodic tilings, but also on colored quasiperiodic tilings~\cite{MP94,L97,B97b,BDLPF13}.
    
    Recall that a CSL is the intersection of a lattice with a rotated copy of itself,
		while its coincidence index is the ratio of the volumes of a (primitive) unit cell of the CSL and a unit cell of the parent lattice. 
		Thus, the coincidence index is a measure of how the two lattices (the lattice and its rotated copy) fit together. 
		Moreover, it also describes how well two crystal grains fit together. 
		Here, it is not so important whether the actual atomic positions at the grain boundary do coincide, 
		as the atoms will rearrange anyhow to minimize energy. 
		What is important is that the grain boundary has a small (two-dimensional) unit cell. 
		The latter is related to the coincidence index and can be calculated by considering the corresponding CSL.
		By a rule of thumb, grain boundaries with a small coincidence index have low energy and are thus preferred. 
		Of course, this rule of thumb is based on geometric considerations only and cannot replace an actual calculation of the energy of the grain boundary.
		Nevertheless, it serves as an important tool in the theory of interfaces as it gives a good hint on which grain boundaries are preferred (\emph{cf.}~\cite{H97,SB2006}).
		
		Hence, we consider coincidences of lattices only, although coincidences of actual crystals could be considered as well and have been considered~\cite{LZ14}. 
		This means that we restrict our attention to coincidence isometries that fix the origin, 
		but see~\cite{LZ14} for a treatment of general affine coincidence isometries.
                
		Note that most results of the present paper have been announced earlier in~\cite{LZ11} (but without proofs).  
		That earlier article,  where examples using the square lattice and the Ammann-Beenker tiling are depicted, 
		was intended more for experimentalists and crystallographers, while this contribution is more for mathematical crystallographers.
		Here, we illustrate the general results through examples involving the cubic lattices in dimension $d=3$ and involving the hypercubic lattices in $d=4$.
		The cubic lattices are certainly important lattices, 
		and our example will also clarify the connection of our approach with colorings to the shelling argument of Grimmer \emph{et al.}~\cite{GBW74}. 
		The example involving the hypercubic lattices in $d=4$ serves to show that our approach is not only
		confined to periodic crystals but is also applicable for quasicrystals.
		Recall that the lattice $\Z^4$ can be used to generate the Ammann-Beenker tiling via the cut-and-projection scheme (\emph{cf.}~\cite{TAO1}).
		Even though both examples can be solved by means of orthogonal matrices, 
		the most efficient way to solve them is to make use of the Hurwitz quaternions. 
		Quaternions are so powerful and natural in this context that they have been used for CSLs since at least 1974 by Grimmer~\cite{G74}
		(with Ranganathan's paper~\cite{R66} containing them already in some disguised form), and have become a standard tool~\cite{SB2006} in this area.
		The details of these calculations are contained in the appendix, to make the main results more accessible for those who are not so familiar with quaternions. 
	
	\section{Preliminaries}
		Let us briefly recall some basic notions first, for more details see~\cite{B97}.  A \emph{lattice} $\G$ (of \emph{rank} and \emph{dimension} $d$) is a discrete subset of 
		$\Rd$ that is the \Z-span of $d$ linearly independent vectors $v_1,\ldots,v_d$ over $\R$.  The set $\set{v_1,\ldots,v_d}$ is referred to as a \emph{basis} of $\G$. As a group, 
		$\G$ is isomorphic to the free abelian group of rank $d$.  A \emph{sublattice} $\G'$ of $\G$ is a subset of $\G$	that forms a subgroup of finite index in $\G$.  This means that 
		$\G'$ is itself a lattice and is of the same rank and dimension as $\G$.  Here, the group	index $[\G:\G']<\infty$ can be interpreted geometrically as the ratio of the 
		volume of a fundamental domain of $\G'$ by the volume of a fundamental domain of $\G$.	
		
		The \emph{dual} of a lattice $\G$ is the lattice 
		\[\G^*\vcentcolon=\set{x\in\Rd:\inn{x}{y}\in\Z \text{ for all }y\in\G},\] 
		where $\inn{\cdot}{\cdot}$ denotes 
		the standard scalar product in $\Rd$.  Two lattices $\G_1$ and $\G_2$ are said to be \emph{commensurate} if their intersection is a sublattice of both 
		lattices.  The \emph{sum} of two commensurate lattices $\G_1$ and $\G_2$, 
		\[\G_1+\G_2\vcentcolon=\set{\ell_1+\ell_2:\ell_1\in\G_1,\ell_2\in\G_2},\]
		also forms a lattice.
		
		An orthogonal transformation $R\in\operatorname{O}(d)\vcentcolon=\operatorname{O}(d,\R)$ is said to be a \emph{coincidence isometry} of the lattice $\G$ if $\G\cap R\G$ 
		is a sublattice of both $\G$ and $R\G$.  The sublattice $\G(R)\vcentcolon=\G\cap R\G$ is called the \emph{coincidence site lattice} (CSL) of $\G$ generated by $R$, and 
		\[\S_{\G}(R)\vcentcolon=[\G:\G(R)]=[R\G:\G(R)]\] 
		is called the \emph{coincidence index of $R$ with respect to $\G$}.  The set of coincidence isometries of $\G$
		forms a group, denoted by $\OC(\G)$~\cite[Theorem 2.1]{B97}.  
    Note that $\OC(\G)$ is a countable subgroup of $\OG(d)$, 
		and it contains the \emph{point group} $P(\G)$ of $\G$ as a subgroup. 
		In particular, we can characterize $P(\G)$ by means of the coincidence index as follows:
		\[
			P(\G)=\set{R\in \operatorname{O}(d):R\G=\G}
			=\set{R\in \OC(\G)\mid \S_{\G}(R)=1}.
		\]
    The subgroup of $\OC(\G)$ formed by all coincidence rotations of $\G$ is written as $\SOC(\G):=\OC(\G)\cap\SO(d)$.	

		Coincidences of a lattice and its sublattices are closely related.  In particular, we have the following result~\cite{B97}.
		\begin{fact}\label{ocsub}
			If $\G_2$ is a sublattice of $\G_1$, then $\OC(\G_1)=\OC(\G_2)$.
		\end{fact}
		
		For instance, the three cubic lattices share the same coincidence isometries.  The same holds for the two four-dimensional hypercubic lattices.  Using
		a shelling argument~\cite{GBW74}, it is known that the coincidence index for a given coincidence isometry is the same for all three cubic lattices (see
		Fact~\ref{coincindcubic}).  However, this is clearly not the case for the four-dimensional hypercubic lattices because they have different point groups, and
		moreover, do not share the same set of coincidence indices~\cite{B97,Z06}.  
		
		In general, even though the groups of coincidence isometries of a lattice $\G_1$ and a sublattice $\G_2$ are the same, the coincidence indices and corresponding 
		multiplicities with respect to $\G_1$ and $\G_2$ may be different.  Let $R\in \OC(\G_1)=\OC(\G_2)$, and denote by $\S_1(R)$ and $\S_2(R)$ the coincidence
		indices of $R$ with respect to $\G_1$ and $\G_2$, respectively.  From the inclusions $\G_2(R)\subseteq\G_2\subseteq\G_1$ and $\G_2(R)\subseteq\G_1(R)\subseteq \G_1$ 
		we can infer the following result~\cite{B97}.  Here, we denote $a$ divides $b$ by $a\mid b$.
		
		\begin{fact}\label{coinindsubknown}
			Let $\G_2$ be a sublattice of $\G_1$ of index $m$.  If $R\in \OC(\G_1)$, then $\S_1(R)\mid m\S_2(R)$.
		\end{fact}	
		
		\noindent By considering the so-called dual lattices of $\G_1$ and $\G_2$, one also obtains 
		\begin{equation}\label{div2}
			\S_2(R)\mid m\S_1(R).
		\end{equation}
		Both divisibility conditions imply the well-known bound on $\S_2(R)$: 
		\begin{equation}\label{bound}
			\tfrac{1}{m}\S_1(R)\leq \S_2(R)\leq m\S_1(R).
		\end{equation}
		In Sections~\ref{coincindexwrt} and ~\ref{colcoin}, we characterize
		when $\S_2(R)=\S_1(R)$ and give a formula for $\S_2(R)$ in terms of $\S_1(R)$ by looking at certain colorings of lattices.
		
		A \emph{coloring} of the lattice $\G_1$ by $m$ colors is an onto mapping $c:\G_1\rightarrow C$, where $C$ is a set of $m$ elements, called 
		colors~\cite{S84}.  If $\G_2$ is a sublattice of
		$\G_1$ of index $m$, then a \emph{coloring of $\G_1$ determined by $\G_2$} is a coloring of $\G_1$ by $m$ colors wherein two points of $\G_1$ are assigned
		the same color if and only if they belong to the same coset of $\G_2$.  In this case, the set of colors $C$ can be identified with the quotient group
		$\G_1/\G_2$ so that the color mapping $c$ is simply the canonical projection of $\G_1$ onto $\G_1/\G_2$ whose kernel is $\G_2$.  
		
		Denote by $G$ the symmetry group of $\G_1$ and fix a coloring $c$ of $\G_1$.  A symmetry $h$ in $G$ is said to be a \emph{color symmetry} of $c$ if it 
		permutes the colors in the coloring, that is, all and only those elements of $\G_1$ having the same color are mapped by $h$ to a fixed color.  We denote the set of
		all color symmetries of $c$ by $H_c$, that is, 
		\[H_c=\set{h\in G\mid\exists\,\sigma_h\in S_C,\forall\,\ell\in\G_1, c(h(\ell))=\sigma_h(c(\ell))},\]
		where $S_C$ is the group of permutations on the set of colors $C$.  The set $H_c$ forms a group and is called the \emph{color group} or \emph{color symmetry group} of $c$
		~\cite{S84,GS87,DLPF07}.  The mapping $P:H_c\rightarrow S_C$ with $h\mapsto\sigma_h$ defines a group homomorphism, and thus the group $H_c$ acts on $C$.  The
		kernel of $P$, \[K_c=\set{k\in H_c:c(k(\ell))=c(\ell),\,\forall\,\ell\in\G_1},\] 
		is the subgroup of $H_c$ whose elements fix the colors in $c$.  In other words, $K_c$ is the symmetry group of the colored lattice.  
		
		We shall generalize the notion of a color symmetry to that of a color coincidence.  Also, we associate the property that a coincidence isometry $R$
		of $\G_1$ is a color coincidence of the coloring of $\G_1$ determined by a sublattice of $\G_2$ with a relationship between $\S_1(R)$ and $\S_2(R)$.

	\section{Coincidence index with respect to a sublattice}\label{coincindexwrt}
		Our first goal is to give a relationship between the coincidence indices with respect to a lattice and a sublattice by means of the coloring of the lattice 
		induced by the sublattice.		
		Unless otherwise stated, $\G_1$ is taken throughout to be a lattice having $\G_2$ as a sublattice of index $m$, and we write
		$\G_1=\bigcup_{j=0}^{m-1}(c_j+\G_2)$ with $c_0\vcentcolon=0$.  Consider the coloring of $\G_1$ determined by $\G_2$, where we label the color of the
		coset	$c_j+\G_2$ by $c_j$.

		Fix an $R\in \OC(\G_1)=\OC(\G_2)$.  Consider the following subgroups of $\G_1/\G_2$\,:
		\begin{equation}\label{JK}
			\begin{aligned}
				J&\vcentcolon=\set{c_j+\G_2\in\G_1/\G_2:(c_j+\G_2)\cap\G_1(R^{-1})\neq \varnothing},\\
				K&\vcentcolon=\set{c_k+\G_2\in\G_1/\G_2:(c_k+\G_2)\cap\G_1(R)\neq\varnothing}
			\end{aligned}
		\end{equation}
		The subgroups $J$ and $K$ are nonempty because both have $\G_2=c_0+\G_2$ as an element.  They induce partitions of $\G_1(R^{-1})$ and $\G_1(R)$, respectively,
		given by 
		\begin{equation}\label{partJK}
		 	\begin{aligned}
			 	\G_1(R^{-1})&=\bigcup_{c_j+\G_2\in J}{(c_j+\G_2)\cap\G_1(R^{-1})}\text{ and}\\
			 	\G_1(R)&=\bigcup_{c_k+\G_2\in K}{(c_k+\G_2)\cap\G_1(R)}.
		 	\end{aligned}
		\end{equation}
		The partitions in Eq.~\eqref{partJK} can be thought of as colorings of $\G_1(R^{-1})$ and $\G_1(R)$, respectively, where the colors are inherited from the
		coloring of $\G_1$ determined by $\G_2$.  We shall refer to these colorings as the \emph{induced colorings of $\G_1(R^{-1})$ and $\G_1(R)$ by $c$}. The
		set of colors in the colorings of $\G_1(R^{-1})$ and $\G_1(R)$ are
		\begin{equation}\label{CRinvR}
			\begin{aligned}
				C_{R^{-1}}&\vcentcolon=\set{c_j:c_j+\G_2\in J}\text{ and}\\
				C_R&\vcentcolon=\set{c_k:c_k+\G_2\in K},
			\end{aligned}
		\end{equation}
		respectively.  In addition, we may assume that the coset representatives of $\G_2$ in $\G_1$ are chosen so that $c_j\in\G_1(R^{-1})$ whenever $c_j\in C_{R^{-1}}$,
		and $c_k\in\G_1(R)$ whenever $c_k\in C_R$. 
		
		In fact, there are two possible colorings of $\G_1(R)$.  On the one hand, we have the coloring of $\G_1(R)$ induced by $c$ with colors from $C_R$.  On the other hand,
		a coloring of $R\G_1(R^{-1})=\G_1(R)$ is obtained when $R$ is applied to the induced coloring of $\G_1(R^{-1})$ by $c$ with colors from $C_{R^{-1}}$.  We compare both 
		colorings of $\G_1(R)$ by looking at the relation $\sigma_R$ from $C_{R^{-1}}$ to $C_R$ given by
		\begin{equation}\label{sigma}
			\sigma_R=\big\{(c_j,c_k)\in C_{R^{-1}}\times C_R:R[(c_j+\G_2)\cap\G_1(R^{-1})]\cap
			[(c_k+\G_2)\cap\G_1(R)]\neq\varnothing\big\}.
		\end{equation}
		Here, $(c_j,c_k)\in\sigma_R$ means that some of the points colored $c_j$ in the coloring of $\G_1(R^{-1})$ are mapped by $R$ onto points with color $c_k$ in
		the coloring of $\G_1(R)$.  
	
    Note that 
		\begin{equation}\label{intandeqthm}
			[\G_2\cap\G_1(R)]\cap R[\G_2\cap\G_1(R^{-1})]=\G_2(R),							
		\end{equation}
		which tells us that $(c_0,c_0)\in\sigma_R$ and so $\sigma_R$ is always nonempty~\cite{LZ11}.  Moreover, $\G_2(R)$ is made up of exactly those
		points colored $c_0$ in the coloring of $\G_1(R)$ whose preimages under $R$ are also points colored $c_0$ in the coloring of $\G_1(R^{-1})$.
		
		\begin{figure}[ht]
			\centering
			\includegraphics{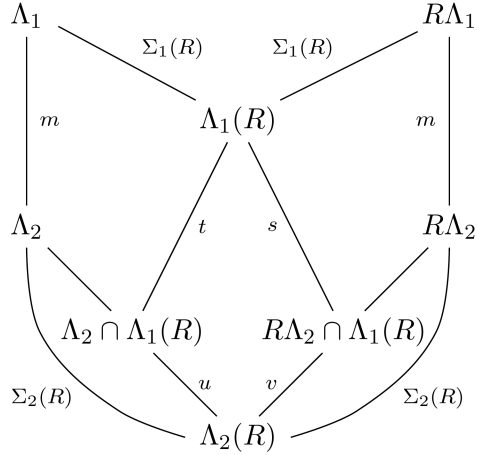}
			\caption{\label{latticediag}Subgroup diagram showing the relationship between $\G_1, \G_2(R)$ and other sublattices together with their indices}
		\end{figure}

		See Figure~\ref{latticediag} for the relationships among the various lattices and the corresponding indices.  In particular, we define
		\begin{equation}\label{stuv}	
			\begin{aligned}
				s&\vcentcolon=[\G_1(R):R\G_2\cap\G_1(R)],\\
				t&\vcentcolon=[\G_1(R):\G_2\cap\G_1(R)],\\
				u&\vcentcolon=[\G_2\cap\G_1(R):\G_2(R)],\\
				v&\vcentcolon=[R\G_2\cap\G_1(R):\G_2(R)].
			\end{aligned}
		\end{equation}

		For further reference, we note some relationships between sublattices of a lattice $\G_1$.

		\begin{lemma}\label{SIT}
			Let $\G_2$ and $\G_2'$ be sublattices of the lattice $\G_1$.  Then the following holds.
			\begin{enumerate}
				\item $[\G_2':\G_2\cap\G_2']=
				\abs{\set{\ell+\G_2\in\G_1/\G_2:(\ell+\G_2)\cap\G_2'\neq\varnothing}}
				$
	
				\item $[\G_2':\G_2\cap\G_2']$ divides $[\G_1:\G_2]$

				\item If $(\ell+\G_2)\cap\G_2'\neq\varnothing$ then $(\ell+\G_2)\cap\G_2'$ is a coset of $\G_2\cap\G_2'$.
			\end{enumerate}
		\end{lemma}
		\begin{proof}
			The second isomorphism theorem for groups implies 
			\[
				\G_2'/(\G_2\cap\G_2')\cong(\G_2+\G_2')/\G_2
				=\set{\ell+\G_2\in\G_1/\G_2:(\ell+\G_2)\cap\G_2'\neq\varnothing},
			\]
			which proves (i).  Since $\G_2+\G_2'$ is a sublattice of $\G_1$, $(\G_2+\G_2')/\G_2$ is a subgroup of $\G_1/\G_2$.  Thus, ${[\G_2+\G_2':\G_2]}=
			[\G_2':\G_2\cap\G_2']$ divides $[\G_1:\G_2]$ by Lagrange's Theorem.  The last statement is clear.
		\end{proof}

		Using Lemma~\ref{SIT}, we can now give restrictions on the values of $s$, $t$, $u$, and $v$, as well as interpretations of these values in relation to the
		colorings of $\G_1(R^{-1})$ and $\G_1(R)$ determined by $\G_2$.  These results are explicitly stated in the following theorem, which was announced in
		\cite{LZ11} without proof.

		\begin{thm}\label{coloring}
			Let $\G_2$ be a sublattice of $\G_1$ of index $m$, $R\in \OC(\G_1)$, and $s,t,u,v$ be as in Eq.~\eqref{stuv}.
			\begin{enumerate}
				\item Then $s\mid m$, $t\mid m$, $u\mid s$, $v\mid t$, and
					\begin{equation}\label{s2r}
						\S_2(R)=\frac{t\cdot u\cdot\S_1(R)}{m}=\frac{s\cdot v\cdot\S_1(R)}{m}.
					\end{equation}
			
				\item Consider the coloring of $\G_1$ determined by $\G_2$ where each coset $c_j+\G_2$ is assigned the color $c_j$ for $0\leq j\leq m-1$, with $c_0=0$. 
				Then
					\begin{enumerate}
						\item $s$ and $t$ are the number of colors in the coloring of $\G_1(R^{-1})$ and $\G_1(R)$, respectively. 
						
						\item $u$ is the number of colors $c_j$ with the property that some points of $\G_1(R^{-1})$ colored $c_j$ are mapped by $R$ to points
						colored $c_0$ in the coloring of $\G_1(R)$.
						
						\item $v$ is the number of colors $c_k$ with the property that some points of $\G_1(R^{-1})$ colored $c_0$ are mapped by $R$ to points 
						colored $c_k$ in the coloring of $\G_1(R)$.
					\end{enumerate}
			\end{enumerate}
		\end{thm}
		\begin{proof}
			Comparing indices in Figure~\ref{latticediag} gives Eq.~\eqref{s2r}.

			If we apply Lemma~\ref{SIT} to the sublattices $\G_2$ and $\G_1(R)$ of $\G_1$, we obtain that $t=\abs{K}=\abs{C_R}$ (see Eq.~\eqref{JK} and 
			Eq.~\eqref{CRinvR}) and $t\mid m$.  Corresponding statements for $s$ are similarly proved by looking at the sublattices $R\G_2$ and $\G_1(R)$ of $R\G_1$. 
			
      To prove that $u\mid s$ and $v\mid t$, we characterize $u$ and $v$ by appropriate groups. To this end, we first mention that
			Lemma~\ref{SIT} implies that for all $c_j+\G_2\in J$ and $c_k+\G_2\in K$,
			\begin{equation}\label{JK2}
				\begin{aligned}
					R[(c_j+\G_2)\cap\G_1(R^{-1})]&=Rc_j+[R\G_2\cap\G_1(R)]\text{ and}\\
					(c_k+\G_2)\cap\G_1(R)&=c_k+[\G_2\cap\G_1(R)].
				\end{aligned}
			\end{equation}
			Let us consider the groups
			\begin{alignat*}{2}
				D&\vcentcolon=
					\{c_j+\G_2\in J: (Rc_j+[R\G_2\cap\G_1(R)])\cap
					[\G_2\cap\G_1(R)]\neq\varnothing\} \subseteq J\text{ and}
				\\
				E&\vcentcolon=
					\{c_k+\G_2\in K: (c_k+[\G_2\cap\G_1(R)])\cap
					[R\G_2\cap\G_1(R)]\neq\varnothing\} \subseteq K.
			\end{alignat*}
      Both groups are nonempty because of Eq.~\eqref{intandeqthm}.  Invoking Lemma~\ref{SIT} to the sublattices $\G_2\cap\G_1(R)$ and
			$R\G_2\cap\G_1(R)$ of $\G_1(R)$ shows that $u=\abs{D}$ and $u\mid s$. Similarly, $v=\abs{E}$ and $v\mid t$.  The interpretations of $u$ and $v$ 
			via colorings follow from $\abs{D}=\abs{\set{c_j:(c_j,c_0)\in\sigma_R}}$ and $\abs{E}=\abs{\set{c_k:(c_0,c_k)\in\sigma_R}}$ (see Eq.~\eqref{sigma} and Eq.~\eqref{JK2}).  
		\end{proof}

		Note that Eq.~\eqref{intandeqthm} and Lemma~\ref{SIT} imply that
		\begin{align*}
			(Rc_j+[R\G_2\cap\G_1(R)])\cap[\G_2\cap\G_1(R)]&=Rc_j+\G_2(R), \\ 
                \intertext{and}
			(c_k+[\G_2\cap\G_1(R)])\cap[R\G_2\cap\G_1(R)]&=c_k+\G_2(R),
		\end{align*}
		for all $c_j+\G_2 \in D$ whenever $Rc_j\in\G_2\cap\G_1(R)$, and $c_k+\G_2 \in E$ whenever $c_k\in R\G_2\cap\G_1(R)$.

		An immediate consequence of Theorem~\ref{coloring} is Fact~\ref{coinindsubknown}, \eqref{div2}, and inequality \eqref{bound}.  
		Notice that the divisibility condition in \eqref{div2}, and hence, the inequality $\S_2(R)\leq m\,\S_1(R)$, was obtained here
		without going through the dual lattices of $\G_1$ and $\G_2$.  This means that the said divisibility condition is true not only for lattices but also for
		$\Z$-modules (see~\cite{B97}), where~\eqref{div2} cannot be obtained by a duality argument because the notion of a dual module is in general not defined.
		
		An explicit example involving the square lattice and the Ammann-Beenker tiling can be found in~\cite{LZ11}.

	\section{Color coincidence}\label{colcoin}
    We have seen from the preceding section that $\S_2(R)$ can be characterized by colorings of $\G_1(R^{-1})$ and $\G_1(R)$.  However, the 
		corresponding groups $J,K,D$, and $E$ are not explicitly known in general.  Nonetheless, the problem simplifies in certain situations, which motivates the following definition.
 		
		\begin{defi}
			Let $\G_2$ be a sublattice of $\G_1$, and write $\G_1=\bigcup_{j=0}^{m-1}(c_j+\G_2)$ with $c_0=0$.  A coincidence isometry $R$ of $\G_1$ is said 
			to be a \emph{color coincidence} of the coloring of $\G_1$ determined by $\G_2$ if for any $c_j\in C_{R^{-1}}$ there exists a $c_k \in C_R$ such that 
			\begin{equation}\label{colcoinc}
				R[(c_j+\G_2)\cap\G_1(R^{-1})]=(c_k+\G_2)\cap\G_1(R).
			\end{equation}
		\end{defi}
	
		Eq.~\eqref{colcoinc} means that all points colored $c_j$ in the coloring of $\G_1(R^{-1})$ are mapped by $R$ to points colored
		$c_k$ in the coloring of $\G_1(R)$.  Thus, if $R\in\OC(\G_1)$ is a color coincidence, it defines an injection $\sigma_R$ from 
		$C_{R^{-1}}$ into $C_R$ via Eq.~\eqref{colcoinc}.  Whenever Eq.~\eqref{colcoinc} is satisfied, $R$ is said to send or map color 
		$c_j$ to color $c_k$.  Furthermore, we say that $R$ \emph{fixes the color $c_j$} if $R$ maps color $c_j$ onto itself. 
		
		In fact, the injection $\sigma_R$ is a bijection, which is seen as follows. By definition, $R$ sends the color $c_0=0\in C_{R^{-1}}$ 
		to exactly one color $c_k\in C_R$, that is,
		\[R[\G_2\cap\G_1(R^{-1})]=R\G_2\cap\G_1(R) =(c_k+\G_2)\cap\G_1(R).\]
		Since $0\in\G_2\cap\G_1(R^{-1})$ and $R(0)=0$, we obtain $c_k=c_0$. In other words,
		$R$ fixes the color $c_0$.  Thus, we have 
		$R[\G_2\cap\G_1(R^{-1})]=\G_2\cap\G_1(R)$, which implies that $\G_2\cap\G_1(R^{-1})$
		and $\G_2\cap\G_1(R)$ have the same index in $\G_1(R)$. It now follows from Theorem~\ref{coloring}
		that the colorings of $\G_1(R^{-1})$ and $\G_1(R)$ have the same number of colors, that is, 
		$\abs{C_{R^{-1}}}=\abs{C_R}$.  Hence $\sigma_R$ is indeed a bijection.
		Observe that this bijection is exactly the relation $\sigma_R$
		defined in Eq.~\eqref{sigma}. This means that an $R\in \OC(\G_1)$ is a color coincidence if and only if the
		relation $\sigma_R$ from $C_{R^{-1}}$ to $C_R$ in Eq.~\eqref{sigma} is a bijection.  
		In particular, if $R$ is a color coincidence and $C_{R^{-1}}=C_R$, then
		$\sigma_R$ is a permutation on $C_{R^{-1}}$.  Thus, a color coincidence $R\in P(\G_1)$ (that is, when $\G_1(R^{-1})=\G_1(R)=\G_1$) is a color symmetry of the
		coloring of $\G_1$. 
		 	
		We have seen that $R$ fixes the color $c_0$ if $R\in\OC(\G_1)$ is a color coincidence.
		Actually, the converse holds as well, as was mentioned in~\cite{LZ11} without proof, and which will be proved below.
		This generalizes the fact that a symmetry operation $R$ of an uncolored lattice $\G$ is a color symmetry of a sublattice coloring of $\G$ if and only if
		$R$ leaves the sublattice invariant (\emph{cf.}~\cite[Theorem 2(i)]{DLPF07}).	
		
		\begin{thm}\label{colcoinequiv}
			Let $\G_2$ be a sublattice of $\G_1$ of index $m$, with $\G_1=\bigcup_{j=0}^{m-1}(c_j+\G_2)$ where $c_0=0$. Then $R\in \OC(\G_1)$ is a color coincidence of
			the coloring of $\G_1$ determined by $\G_2$ if and only if $R$ fixes color $c_0$.
		\end{thm}
		\begin{proof}
			It remains to show the ``if'' part. Suppose $R$ fixes color $c_0$, thus, ${R[\G_2\cap\G_1(R^{-1})]}={\G_2\cap\G_1(R)}$.  This implies that 
			$R\G_2\cap\G_1(R)=\G_2\cap\G_1(R)=\G_2(R)$ by Eq.~\eqref{intandeqthm}.  Hence, $u=v=1$ and $s=t$ by Eq.~\eqref{stuv}.  From Theorem~\ref{coloring}, the 
			colorings of $\G_1(R)$ and $\G_1(R^{-1})$ determined by $\G_2$ must have equal number of colors.  Under the assumption that $c_j\in\G_1(R^{-1})$ for each
			$c_j+\G_2\in J$, it then follows from Eq.~\eqref{JK2} that
			\begin{alignat*}{2}
				R[(c_j+\G_2)\cap\G_1(R^{-1})]&=Rc_j+[R\G_2\cap\G_1(R)]\\
				&=Rc_j+[\G_2\cap\G_1(R)]\\
				&=(Rc_j+\G_2)\cap\G_1(R).
			\end{alignat*}
			This means that $Rc_j+\G_2$ must be one of the cosets $c_k+\G_2\in K$.  Therefore, $R$ is	a color coincidence of the coloring of $\G_1$.
		\end{proof}
 		
		By Theorem~\ref{colcoinequiv}, it is then sufficient to check whether $R[\G_2\cap\G_1(R^{-1})]=\G_2\cap\G_1(R)$ is satisfied to identify if $R\in \OC(\G_1)$ is
		a color coincidence.

		Let $R$ be a color coincidence.  Equations~\eqref{JK2}, \eqref{partJK}, \eqref{intandeqthm}, and 
		Theorem~\ref{colcoinequiv}, yield the following coset decompositions of $\G_1(R)$ with respect to $\G_2(R)$:
		\[\G_1(R)=\bigcup_{c_k+\G_2\in K}[c_k+\G_2(R)]=\bigcup_{c_j+\G_2\in J}[Rc_j+\G_2(R)]\]
		In the language of colorings, this means that a color coincidence $R$ determines a permutation on the set of cosets of $\G_2(R)$ in $\G_1(R)$.  
		Here, $Rc_j+\G_2(R)=c_k+\G_2(R)$ if and only if $R$ sends color $c_j$ to $c_k$. 

    Theorem~\ref{coloring} is simpler for color coincidences since $u=v=1$. In particular, we have the following result.
		\begin{cor}\label{divandeq1}
			If $R$ is a color coincidence of the coloring of the lattice $\G_1$ determined by the sublattice $\G_2$, then $\S_2(R)\mid\S_1(R)$.
		\end{cor}		
			
    If the colorings of $\G_1(R)$ and $\G_1(R^{-1})$ contain all $m$ colors, that is, if $s=t=m$, then we get the following criterion for color coincidences.
		\begin{cor}\label{divandeq2}
			Let $\G_1$ be a lattice having $\G_2$ as a sublattice of index $m$. If $s=t=m$ in Eq.~\eqref{stuv}, then $R\in \OC(\G_1)$ is a color coincidence of the coloring 
			of $\G_1$ induced by $\G_2$ if and only if $\S_2(R)=\S_1(R)$.
		\end{cor}
		
		Denote by $H_{\G_1,\G_2}$ the set of all color coincidences of the coloring of a lattice $\G_1$ determined by a sublattice $\G_2$. Clearly, the identity
		isometry is in $H_{\G_1,\G_2}$.  In addition, it follows from the definition of a color coincidence that if $R\in H_{\G_1,\G_2}$, then so is $R^{-1}$.  The
		question of whether the product of two color coincidences is again a color coincidence, and in effect, whether $H_{\G_1,\G_2}$ forms a group, is yet to be
		resolved.  A step in answering this question is the following proposition.

		\begin{prop}\label{relprimecolcoin}
			Let $\G_2$ be a sublattice of $\G_1$ of index $m$, and $R_1, R_2\in H_{\G_1,\G_2}$.  If $\S_1(R_1)$ is relatively prime to $\S_1(R_2)$, then $R_2R_1\in H_{\G_1,\G_2}$.
		\end{prop}
		\begin{proof}
			Since $\S_1(R_1)$ and $\S_1(R_2)$ are relatively prime, $\G_1(R_2R_1)=\G_1\cap R_2\G_1\cap R_2R_1\G_1$ by~\cite[Corollary 3]{Z10}.  Applying
			Theorem~\ref{colcoinequiv} and Eq.~\eqref{intandeqthm} to $R_1$, one obtains \[R_1\G_2\cap\G_1=R_1\G_2\cap\G_1(R_1)=\G_2\cap\G_1(R_1)=\G_2\cap
			R_1\G_1.\] Similarly, $R_2\G_2\cap\G_1=\G_2\cap R_2\G_1$.  These three equalities yield 
			\[R_2R_1\G_2\cap\G_1(R_2R_1)=\G_2\cap\G_1(R_2R_1).\]
			Thus, $R_2R_1$ fixes color $c_0$, which means that $R_2R_1\in H_{\G_1,\G_2}$ by Theorem~\ref{colcoinequiv}.
		\end{proof}

	\section{Special Cases}	
		We now apply the theory developed in the previous two sections to some special cases.  First, assume that $\G_1(R)\subseteq\G_2$.  Then all points of $\G_1(R)$
		have the same color, or in other words, $t=1$.  Hence, $\S_2(R)\mid\S_1(R)$.  More generally, we have the following result.
		\begin{lemma}\label{spcasecor1}
			Let $\G_2$ be a sublattice of $\G_1$ with $[\G_1:\G_2]=m$, and $R\in \OC(\G_1)$.  If $\G_1(R)$ or $\G_1(R^{-1})$ is a sublattice of $\G_2$, then 
			$\S_2(R)\mid\S_1(R)$.  In particular, both $\G_1(R)$ and $\G_1(R^{-1})$ are sublattices of $\G_2$ if and only if $\S_2(R)=(1/m)\S_1(R)$.
		\end{lemma}
		\begin{proof}
			Both $\G_1(R)$ and $\G_1(R^{-1})$ are sublattices of $\G_2$ if and only if $s=t=u=v=1$.  Applying Eq.~\eqref{s2r} completes the proof.
		\end{proof}

		The possibilities are quite limited when the sublattice $\G_2$ is of prime index in $\G_1$, as can be seen in the next proposition.
		\begin{prop}\label{primeindex}
		 	Suppose $\G_2$ is a sublattice of $\G_1$ of prime index $p$ and $R\in \OC(\G_1)$.
			\begin{enumerate}
				\item If both $\G_1(R)$ and $\G_1(R^{-1})$ are sublattices of $\G_2$ then $\S_2(R)=(1/p)\S_1(R)$.

				\item If exactly one of $\G_1(R)$ and $\G_1(R^{-1})$ is a sublattice of $\G_2$ then $\S_2(R)=\S_1(R)$.

				\item If neither $\G_1(R)$ nor $\G_1(R^{-1})$ is a sublattice of $\G_2$, then $\S_2(R)=\S_1(R)$ whenever $R$ is a color coincidence of
				the coloring of $\G_1$ induced by $\G_2$, and $\S_2(R)=p\,\S_1(R)$ otherwise.
			\end{enumerate}
		\end{prop}
		\begin{proof}
			Statements (i) and (ii) are immediate from Lemma~\ref{spcasecor1}.  If neither $\G_1(R)$ nor $\G_1(R^{-1})$ is contained in $\G_2$, then
			$s,t>1$ (see Eq.~\eqref{stuv}).  We have $s=t=p$ by Theorem~\ref{coloring}.  Recall that $R$ is a color coincidence of the coloring of $\G_1$ if and only if
			$u=v=1$ by Theorem~\ref{colcoinequiv} and Eq.~\eqref{intandeqthm}.  Thus, $\S_2(R)=\S_1(R)$ whenever $R$ is a color coincidence of the coloring of $\G_1$ by
			Eq.~\eqref{s2r}.  Otherwise, $u=v=p$ by Theorem~\ref{coloring}, and it follows from Eq.~\eqref{s2r} that $\S_2(R)=p\,\S_1(R)$.
		\end{proof}

		We can say more if $\S_1(R)$ and $m$ are relatively prime.
		\begin{prop}\label{relprime}
		 	Let $\G_2$ be a sublattice of $\G_1$ with $[\G_1:\G_2]=m$, and $R\in \OC(\G_1)$.  If $\S_1(R)$ and $m$ are relatively prime, then all colors in the coloring 
		 	of $\G_1$ determined by $\G_2$ appear in both colorings of $\G_1(R)$ and $\G_1(R^{-1})$, and $\S_2(R)\mid\S_1(R)$.
		\end{prop}
		\begin{proof}
			From \eqref{s2r}, $(t/m)\S_1(R)=(1/u)\S_2(R)\in\N$ because $u\mid\S_2(R)$ (see Figure~\ref{latticediag}).  Since $\S_1(R)$ is relatively prime 
			to $m$, then $m\mid t$. However, $t\mid m$ as well by Theorem~\ref{coloring} and so $t=m$.  Similar arguments yield $s=m$.  Finally, $\S_2(R)\mid\S_1(R)$
			follows from Eq.~\eqref{s2r}.
		\end{proof}
		
	\section{Examples}
		We now give examples involving the cubic and hypercubic lattices.  For this, we parametrize linear isometries in three and four dimensions by quaternions.  
		Before we proceed, we recall several properties of the quaternion algebra $\H(\R)$.  
		
		Let $\set{\e,\ii,\j,\k}$ be the standard basis of $\R^4$ where $\e={(1,0,0,0)}^T$, $\ii={(0,1,0,0)}^T$, $\j={(0,0,1,0)}^T$, and $\k={(0,0,0,1)}^T$.  The
		\emph{quaternion algebra} is the algebra $\H\vcentcolon=\H(\R)=\R\e+\R\ii+\R\j+\R\k\cong\R^4$, whose elements are called \emph{quaternions} and are
		written as either $q=q_0\e+q_1\ii+q_2\j+q_3\k$ or $q=(q_0,q_1,q_2,q_3)$.  The product of two quaternions $q=(q_0,q_1,q_2,q_3)$ and $p=(p_0,p_1,p_2,p_3)$ 
		is given by 
		\begin{multline*}
			q\,p=(q_0p_0-q_1p_1-q_2p_2-q_3p_3)\e+\\(q_0p_1+q_1p_0+q_2p_3-q_3p_2)\ii+(q_0p_2-q_1p_3+q_2p_0+q_3p_1)\j\\+(q_0p_3+q_1p_2-q_2p_1+q_3p_0)\k.
		\end{multline*}
	
		The \emph{conjugate} of a quaternion $q=(q_0,q_1,q_2,q_3)$ is $\overline{q}=(q_0,-q_1,-q_2,-q_3)$, and its \emph{norm} is 
		\[\nrq=q\,\overline{q}=q_0^2+q_1^2+q_2^2+q_3^2\in\R.\]  The \emph{real part} and \emph{imaginary part} of $q$ are defined as $\Re{q}=q_0$ and 
		$\Im{q}=q_1\ii+q_2\j+q_3\k$, respectively. In addition, we define $\Im{\H}\vcentcolon=\set{\Im{q}: q\in\H}$. 
		Every nonzero quaternion $q$ has a multiplicative inverse given by 
		$q^{-1}=\overline{q}/\nrq$, which makes $\H$ an associative division algebra.  

		A quaternion whose components are all integers is called a \emph{Lipschitz quaternion} or \emph{Lipschitz integer}.  The set of Lipschitz quaternions shall
		be denoted by 
		\[\L=\set{(q_0,q_1,q_2,q_3)\in\H:q_0,q_1,q_2,q_3\in\Z}.\]  
		A \emph{primitive quaternion} $q=(q_0,q_1,q_2,q_3)$ is a quaternion in $\L$ whose components are relatively prime, that is, $\gcd(q_0,q_1,q_2,q_3)=1$.  On
		the other hand, a \emph{Hurwitz quaternion} or \emph{Hurwitz integer} is a quaternion whose components are all integers or all half-integers.  The set $\J$
		of Hurwitz quaternions can be written as 
		\[\J=\L\cup[(\tfrac{1}{2},\tfrac{1}{2},\tfrac{1}{2},\tfrac{1}{2})+\L].\]
		For an extensive treatment on quaternions, we refer to \cite{KR91,CS03,H19,HW79}.
		
		It suffices to look at the primitive cubic lattice when studying the coincidences of the three-dimensional cubic lattices because of the following
		well-known result \cite{GBW74, B97}. 
		\begin{fact}\label{coincindcubic}
			Let $\G_P=\Z^3$, $\G_B=\G_P\cup[(\frac{1}{2},\frac{1}{2},\frac{1}{2})+\G_P]$, and $\G_F=\G_B^*$ denote the primitive cubic, body-centered cubic, and 		
			face-centered cubic lattice, respectively.  Then $\OC(\G_P)=\OC(\G_F)=\OC(\G_B)=\operatorname{O}(3,\Q)$.  Moreover, if $R\in \operatorname{O}(3,\Q)$, then
			$\S_{\G_P}(R)=\S_{\G_F}(R)=\S_{\G_B}(R)$.
		\end{fact}		
						
		We now embed the cubic lattices in the Hurwitz ring $\J$ of integer quaternions and we employ \emph{Cayley's parametrization} of $\SO(3)$ (\emph{cf.}~\cite{KR91}).  
		That is, for every $R\in \SOC(\G)=\SO(3,\Q)$, there exists a primitive quaternion $q$ (which is unique up to a sign) so that for all $x\in\Im\H$, 
		$R(x)=qxq^{-1}=qx\overline{q}/\nrq$.  In such a case, we denote $R$ by $R_q$. 
		In fact, the coincidence index of $R_q\in \SOC(\G)$ is given by $\S(R_q)={\nrq}/{2^{\ell}}$, where $\ell$ is the largest power of $2$ that 
		divides $\nrq$ (called the odd part of $\nrq$) \cite{GBW74,G84,B97}.  
		Also, $T\in \OC(\G)\setminus \SO(\G)$ can be written as $T_q=-R_q$, where $R_q\in \SOC(\G)$.  
								
		\begin{ex}\label{p-bcc}
		 	Let $\G_1$ be the body-centered cubic lattice $\G_1=\Im{\J}$ and $\G_2$ its maximal primitive cubic sublattice $\G_2=\Im{\L}$.  Here, $[\G_1:\G_2]=2$ and so 
		 	the coloring of $\G_1$ determined by $\G_2$ consists of two colors.  For each $R\in \OC(\G_1)=\OC(\G_2)$, one has $\S_2(R)=\S_1(R)$ (see
		 	Fact~\ref{coincindcubic}).
		 	
			Since all the coincidence indices of $\G_1$ are odd numbers, both colorings of $\G_1(R)$ and $\G_1(R^{-1})$ induced by $\G_2$ include two colors by
			Proposition \ref{relprime}.  It follows then from Corollary~\ref{divandeq2} that $H_{\G_1,\G_2}$ is the entire group $\OC(\G_1)$ with all coincidence
			isometries fixing both colors.\qed
		\end{ex}

		\begin{ex}\label{bcc-p}
			This time, take the primitive cubic lattice $\G_1=\Im{\L}$ to be the parent lattice and the body-centered cubic lattice $\G_2=2\,\Im{\J}$ to be the
			sublattice of index 4 in $\G_1$.  Write $\G_1=\cup_{j=0}^3(c_j+\G_2)$ where $c_0=0$.  

			Arguments analogous to Example~\ref{p-bcc} yield that all four colors in the coloring of $\G_1$ determined by $\G_2$ show up in the coloring of $\G_1(R)$
			and $\G_1(R^{-1})$, and $H_{\G_1,\G_2}=\OC(\G_1)$.  Indeed, the result that $H_{\G_1,\G_2}=\OC(\G_1)$ in Examples \ref{p-bcc} and \ref{bcc-p} also follows from
			a shelling argument \cite{GBW74}.  That is, since all the points of $\G_1$ that lie on the same shell about the origin are either from $\G_2$ or from $\G_1\setminus\G_2$, 
			then every coincidence isometry of $\G_1$ must be a color coincidence by Theorem~\ref{colcoinequiv}.
		
			We also identify the color permutation induced by each coincidence isometry of $\G_1$.  
			If a coincidence isometry $R\in \OC(\G_1)$ is parametrized by the primitive quaternion $q$, then 
			\begin{enumerate}
				\item $R$ fixes all the colors if and only if $\nrq\equiv 1\imod{4}$,

				\item $R$ fixes two of the colors (one of the fixed colors is $c_0$) if and only if ${\nrq\equiv 2\imod{4}}$, and
		
				\item $R$ fixes only color $c_0$ if and only if $\nrq\equiv 0\imod{4}$.
			\end{enumerate}
			This result is justified in the Appendix.  Furthermore, the set of color permutations generated by the color coincidences of the coloring of $\G_1$ forms a group that 
			is isomorphic to $S_3$. \qed
		\end{ex}
		
    We now look at an example involving lattices in $\R^4$.
		There are only two distinct types of hypercubic lattices in four dimensions, namely the primitive hypercubic lattice $\Z^4$ and the centered hypercubic
		lattice $D_4$ (\emph{cf.}~\cite{CS88,BBNWZ78}).  We identify $D_4$ and $\Z^4$ with the set of Hurwitz quaternions $\J$ and the set of Lipschitz quaternions $\L$, respectively.		
		Even though $\OC(\Z^4)=\OC(D_4)=\operatorname{O}(4,\Q)$, the coincidence indices of a coincidence isometry with respect to the two lattices are not
		necessarily equal~\cite{B97,Z06,BZ08}.  

		Elements of $\operatorname{O}(4)$ can be parametrized by pairs of quaternions. For any $R\in\SO(4)$, there exists a pair of quaternions $(q,p)$ such that 
		$R(x)=qx\overline{p}/\abs{q\,p}$.  In this case we write $R=R_{q,p}$.  Whenever $R_{q,p}\in\SO(4,\Q)$, then $q$ and $p$ can be chosen to be primitive quaternions in $\L$. 
		However, not every pair $(q,p)$ of primitive quaternions defines an $R\in\SO(4,\Q)$. In fact, a primitive pair $(q,p)$ defines an $R\in\SO(4,\Q)$ if and only if $(q,p)$ is
		\emph{admissible}, that is, if $\abs{q\,p}\in\N$. 
                
		The coincidence index of ${R_{q,p}\in \SOC(D_4)}$ is given by $\S_{D_4}(R_{q,p})=\lcm(\nrq/2^k,\nrp/2^{\ell})$, where $k$ and $\ell$ are the highest powers such 
		that $2^k$ and $2^{\ell}$ divide $\nrq$ and $\nrp$, respectively.  On the other hand,  
		\begin{equation}\label{coincindhypcub}
				\S_{\Z^4}(R_{q,p})=\lcm(\S_{D_4}(R_{q,p}),\den(R_{q,p})),
		\end{equation}
     where $\den(R_{q,p})$ is the denominator of $R_{q,p}$, that is, the least common denominator of all entries of $R_{q,p}$ in its canonical matrix representation~\cite{B97}.
		
     Every rotoreflection $T$ in $\operatorname{O}(4,\Q)$ can also be parametrized by an admissible pair of quaternions, namely via 
		$T(x)=T_{q,p}(x)=q\overline{x}\overline{p}/\abs{q\,p}$. In particular, $T_{1,1}(x)=\overline{x}$ and  
    $T_{q,p}=R_{q,p}\cdot T_{1,1}$.  Moreover, we have $\S_{D_4}(T_{q,p})=\S_{D_4}(R_{q,p})$ and $\S_{\Z_4}(T_{q,p})=\S_{\Z_4}(R_{q,p})$.		
		
		The following example examines the set of color coincidences of the coloring of $D_4$ determined by $\Z^4$ which forms a proper subgroup of $\OC(D_4)$.
		
		\begin{ex}\label{hypcubex}
		 	Take $\G_1$ to be the centered hypercubic lattice $\G_1=\J$ and $\G_2$ to be the primitive hypercubic lattice $\G_2=\L$ of index 2 in $\G_1$.  Let 
		 	$R=R_{q,p}\in \SOC(\G_1)$ be parametrized by the admissible pair $(q,p)$ of primitive quaternions.  

			Since $[\G_1:\G_2]=2$ and $\S_1(R)$ is always odd, $s=t=2$ by Proposition~\ref{relprime}.  It follows from Theorem~\ref{coloring} that 
			$\S_2(R)=u\,\S_1(R)$ with $u=1$ or $u=2$. In particular, $\S_2(R)=\S_1(R)$ if and only if $R\in H_{\G_1,\G_2}$ by Corollary~\ref{divandeq2}.  Hence,
			$H_{\G_1,\G_2}=\set{R\in \OC(\G_1):\S_2(R)=\S_1(R)}$.    
						
			The set $H_{\G_1,\G_2}$ in this instance still forms a group.  To see this, write \[H_{\G_1,\G_2}=\set{R\in \OC(\G_1):\S_2(R)\text{ is odd}}.\]  Now, if
			$R_1,R_2\in H_{\G_1,\G_2}$ then $\S_2(R_2R_1)\mid\S_2(R_2)\cdot\S_2(R_1)$ by~\cite[Lemma 1]{Z10}.  This implies that $\S_2(R_2R_1)$ must also be odd, and hence
			$R_2R_1\in H_{\G_1,\G_2}$.  This proves that the product of two color coincidences of the coloring of $\G_1$ is again in $H_{\G_1,\G_2}$, and hence, $H_{\G_1,\G_2}$ is a
			subgroup of $\OC(\G_1)$.

			We identify the conditions that the quaternion pair $(q,p)$ should satisfy so that $R\in H_{\G_1,\G_2}$ in the Appendix.  This yields the following result about the coincidence
			index of a coincidence isometry of $\Z^4$ (compare with Eq.~\eqref{coincindhypcub}).
			\begin{prop}\label{hypcubindicesnew}
				Let $R_{q,p}\in \SOC(\Z^4)$ where $(q,p)$ is an admissible pair of primitive quaternions.  Then 
				either 
				\begin{alignat*}{2}
					\S_{\Z^4}(R_{q,p})&=\S_{D_4}(R_{q,p})=\lcm\lt(\tfrac{\nrq}{2^k},\tfrac{\nrp}{2^{\ell}}\rt)\;\;\text{or}\\
					\S_{\Z^4}(R_{q,p})&=2{\S_{D_4}(R_{q,p})},
				\end{alignat*}
				where $k$ and $\ell$ are the highest powers such that $2^k$ and $2^{\ell}$ 
				divide $\nrq$ and $\nrp$, respectively.  In particular, $\S_{\Z^4}(R_{q,p})=\S_{D_4}(R_{q,p})$  holds if and only if one of the following conditions are
				satisfied:
				\begin{enumerate}
					\item $\nrq$ and $\nrp$ are odd,
					
					\item $\nrq\equiv\nrp\equiv 2\imod{4}$ with $\inn{q}{p}$ even,
	
					\item $\nrq\equiv\nrp\equiv 0\imod{4}$ with $4\mid\inn{q}{p}$.
				\end{enumerate}
				The same holds for $T_{q,p}\in \OC(\Z^4)\setminus \SOC(\Z^4)$.
			\end{prop}
			Note that Proposition~\ref{hypcubindicesnew} is consistent and similar to the conditions set forth in \cite{Z06} on how to identify which of the 576 pure 
			symmetry rotations of $D_4$ are also pure symmetry rotations of $\Z^4$.
		\end{ex}	
		
	\section{Conclusion and Outlook}
		Let $\G_1$ be a lattice and $R\in\OC(\G_1)$.  
		A method to compute for the coincidence index of $R$ with respect to a sublattice $\G_2$ of $\G_1$
		was formulated in Theorem~\ref{coloring} via properties of the coloring of $\G_1$ determined by $\G_2$. 
		This led to the extension of color symmetries to color coincidences of sublattice colorings of $\G_1$.
		Theorem~\ref{colcoinequiv} shows that the color coincidences of a coloring of $\G_1$ induced by $\G_2$ are precisely
		those coincidence isometries of $\G_1$ that leave $\G_2$ invariant. 
		
		The authors suspect that the set of color coincidences of a sublattice coloring of $\G_1$ does not form a group.  
		A counterexample will not only confirm this suspicion, 
		but should also shed more light on the coincidence index of a product of two coincidence isometries.  
		The generalization of a color symmetry to color coincidences is not only interesting in its own right, 
		but it also provides a further connection between the relationship
		of the coincidence indices of a lattice and sublattice. 
		It would also be helpful if a comparable connection between similar isometries (\emph{cf.}~\cite{GB08}) and colorings of lattices can be set up. 
		This would give a unified perspective among similar sublattices (SSLs), CSLs, and colorings of lattices.	
			
	\section*{Acknowledgements}
		M.J.C.~Loquias would like to thank the German Academic Exchange Service (DAAD) for financial support during his stay in Bielefeld.  This work was supported by the 
		German Research Council (DFG), within the CRC 701.
		
	\appendix
	\section*{Appendix}
		
		Here, we give the promised justifications and proofs in Examples~\ref{bcc-p} and \ref{hypcubex}.
		
		\subsection*{Example 2}
			Recall that $\G_1=\cup_{j=0}^3(c_j+\G_2)$, where $\G_1=\Im{\L}$, $\G_2=2\,\Im{\J}$, and $c_0=0$.  Let $R=R_q\in \SOC(\G_1)$, where $q=(q_0,q_1,q_2,q_3)$ is a primitive
			quaternion.  For sure, $R$ is a color coincidence that fixes color $c_0$ because of 
			Theorem~\ref{colcoinequiv}.  We shall now determine how $R$ acts on the other colors $c_1$, $c_2$, $c_3$, and thus, the color permutation that $R$ generates.
			
			Since $R$ is a color coincidence, it is enough to consider a representative from each coset of $\G_2$ in $\G_1$ that is in $\G_1(R^{-1})$, and afterwards 
			identify to which coset of $\G_2$ the representative is sent by $R$.  Take $c_1=\S_1(R)\ii$, $c_2=\S_1(R)\j$, and $c_3=\S_1(R)\k$.  Indeed, for 
			$j\in\set{1,2,3}$, $c_j\notin\G_2$ since $\S_1(R)$ is odd, and $c_j\in\G_1(R^{-1})$ because 
			$\S_1(R)=\den(R)=\gcd\set{k\in\N:k\,R\text{ is an integral matrix}}$ (see~\cite{B97}).  One obtains
			\begin{alignat*}{2}
				R(c_1)&=\tfrac{\S_1(R)}{\nrq}\,q\ii\overline{q},\\
				R(c_2)&=\tfrac{\S_1(R)}{\nrq}\,q\j\overline{q},\\
				R(c_3)&=\tfrac{\S_1(R)}{\nrq}\,q\k\overline{q},
			\end{alignat*}
			where $\S_1(R)/\nrq=1/2^{\ell}$ with $\ell\in\set{0,1,2}$.  This gives rise to three different cases.

			Before we proceed, we take note of the following facts that will be used in the computations thereafter.
			
			\begin{fact}\label{f2}
				For every $q\in\H$ and $x\in\Im{\H}$, we have $qx\overline{q}\in\Im{\H}$.
			\end{fact}
			
			If $q\in\J$ and $j$ is the highest power of 2 such that $2^j$ divides $\nrq$, then $q={(1+\ii)}^jp$ for some $p\in\J$ of odd norm~\cite[page 37]{H19}.
			The next statement follows from this result.
					
			\begin{fact}\label{f1}
				The set ${(1+\ii)}^r\J=\{q\in\J:2^r\mid\nrq\}$ is a two-sided ideal of $\J$ for all $r\in\N$.
			\end{fact}			

			\begin{fact}\label{f3}
				Let $\G_2=2\,\Im{\J}=2\,\Im{\L}\cup [(0,1,1,1)+2\,\Im{\L}]$.  Then the following hold:  
				\begin{enumerate}
					\item $2\J\cap\Im{\H}=2\,\Im{\L}\subseteq\G_2$
					
					\item If $q\in\J$ then $q-\overline{q}\in\G_2$.
				\end{enumerate}  
			\end{fact}
			\begin{proof}
				Statement (i) is clear. If $q\in\L$ then $q-\overline{q}\in\G_2$ by (i).  On the other hand, if $q\in\J\setminus\L$ then 
				$q-\overline{q}\in(0,1,1,1)+2\,\Im{\L}\subseteq\G_2$.  This proves (ii).
			\end{proof}

			The three possible ratios of $\S_1(R)/\nrq$ are investigated below.

			\smallskip\noindent\textsc{Case I:} $\S_1(R)/\nrq=1$, that is, $\nrq\equiv 1\imod{4}$ and either one or three among the components of $q$ is/are odd.
			
				For instance, suppose $q_0$ is odd while $q_1$, $q_2$, $q_3$ are even, or $q_0$ is even while $q_1$, $q_2$, $q_3$ are odd.  In both instances, one can
				write $q=r+s$, where $r\in 2\J$ and $s=\e$.  One obtains
				\[R(c_j)=qx_j\overline{q}=rx_j\overline{r}+rx_j\overline{s}+sx_j\overline{r}+sx_j\overline{s}\]
				where $x_j=\ii,\j,\k$ if $j=1,2,3$, respectively.  Facts~\ref{f2}, \ref{f1}, and \ref{f3}(i) imply that $rx_j\overline{r}\in\G_2$, and
				$rx_j\overline{s}+sx_j\overline{r}=rx_j\overline{s}-\overline{rx_j\overline{s}}\in\G_2$ by Fact~\ref{f3}(ii).  Hence, $R(c_j)\in 
				sx_j\overline{s}+\G_2=x_j+\G_2=c_j+\G_2$ for $j\in\set{1,2,3}$.  Similarly, for the other three possibilities, $R(c_j)\in c_j+\G_2$ for $j\in\set{1,2,3}$. 
				Therefore, in all instances, $R$ fixes all colors.

			\noindent\textsc{Case II:} $\S_1(R)/\nrq=\frac{1}{2}$, that is, $\nrq\equiv 2\imod{4}$ and exactly two components of $q$ are odd.

				Consider the instance when both $q_0$ and $q_1$ are odd, or when both $q_2$ and $q_3$ are odd.  Either way, one can express $q$ as $q=r+s$ where $r\in2\J$ 
				and $s=(1,1,0,0)$.  One has in this case
				\[R(c_j)=\tfrac{1}{2}(rx_j\overline{r}+rx_j\overline{s}+sx_j\overline{r}+sx_j\overline{s})\]
				where $x_j=\ii,\j,\k$ if $j=1,2,3$, respectively.  Now, $\frac{1}{2}rx_j\overline{r}\in 2\J$ because 4 divides $\abss{\frac{1}{2}rx_j\overline{r}}$.  
				This, with Facts~\ref{f2} and \ref{f3}(i), implies that $\frac{1}{2}rx_j\overline{r}\in\G_2$.  Since ${\frac{1}{2}rx_j\overline{s}\in\J}$, one obtains
				that $\frac{1}{2}rx_j\overline{s}+\frac{1}{2}sx_j\overline{r}\in\G_2$ by Fact~\ref{f3}(ii).  Therefore, 
				\[R(c_j)\in\tfrac{1}{2}sx_j\overline{s}+\G_2=\lt\{
				\begin{aligned}
					&c_1+\G_2, &&\text{if }j=1\\
					&c_3+\G_2, &&\text{if }j=2\\
					&c_2+\G_2, &&\text{if }j=3.
				\end{aligned}\rt.\]
				Thus, $R$ induces the permutation $(c_2c_3)$ of colors.   Similar computations for the other two possibilities yield that if $q_0$ and $q_j$ are of the
				same parity, where $j\in\set{1,2,3}$, then $R$ fixes both colors $c_0$ and $c_{j}$ and swaps the other two colors.

			\noindent\textsc{Case III:} $\S_1(R)/\nrq=\frac{1}{4}$, that is, $\nrq\equiv 0\imod{4}$ and all components of $q$ are odd.

				Suppose an even number of components of $q$ are congruent to 1 modulo 4.  Then, one can write $q=r+s$, where $r\in (1,1,0,0)2\J$ and $s=(1,1,1,1)$.  Thus,
				\[R(c_j)=\tfrac{1}{4}(rx_j\overline{r}+rx_j\overline{s}+sx_j\overline{r}+sx_j\overline{s})\]
				where $x_j=\ii,\j,\k$ if $j=1,2,3$, respectively.  Since 4 divides ${\abs{\frac{1}{4}rx_j\overline{r}}}^2$, one has ${\frac{1}{4}rx_j\overline{r}\in 2\J}$ 
				and together with Facts~\ref{f2} and \ref{f3}(i), $\frac{1}{4}rx_j\overline{r}\in\G_2$.  Also, by Fact~\ref{f3}(ii) one concludes that
				$\frac{1}{4}rx_j\overline{s}+\frac{1}{4}sx_j\overline{r}\in\G_2$ because $\frac{1}{4}rx_j\overline{s}\in\J$.  Finally, 
				\[R(c_j)\in\tfrac{1}{4}sx_j\overline{s}+\G_2=\lt\{
				\begin{aligned}
					&c_2+\G_2, &&\text{if }j=1\\
					&c_3+\G_2, &&\text{if }j=2\\
					&c_1+\G_2, &&\text{if }j=3.
				\end{aligned}\rt.\]
				Hence, $R$ generates the permutation $(c_1c_2c_3)$ of colors.  On the other hand, if an odd number of components of $q$ are congruent to 1 modulo 4, then
				similar arguments show that $R$ induces the permutation $(c_1c_3c_2)$ of colors.
				
				Given a coincidence reflection $T_q\in \OC(\G_1)$, the color permutation that it effects is the same as that of the coincidence rotation $R_q$.	
	
			\subsection*{Example 3}
				Recall that $\G_1=\J$, $\G_2=\L$, and $R=R_{q,p}\in \SOC(\G_1)$ is parametrized by the admissible pair $(q,p)$ of primitive quaternions.  
				The following looks at the conditions that the quaternion pair $(q,p)$ should satisfy so that $R\in H_{\G_1,\G_2}$.

				Going through the different possible admissible quaternion pairs $(q,p)$ results in the following cases.  In each case, the sets $R\G_2\cap\G_1(R)$ and
				$\G_2\cap\G_1(R)$ are compared in order to ascertain whether $R\in H_{\G_1,\G_2}$ or not (see Theorem~\ref{colcoinequiv}).

				\smallskip\noindent\textsc{Case I:} $\nrq$ and $\nrp$ are odd

					Suppose $v\in\G_2\cap\G_1(R)$ and write $v=qw\overline{p}/\abs{q\,\overline{p}}$ for some $w\in\J$.  This means that 
					$\abs{q\,\overline{p}}v=qw\overline{p}\in\L$.  Since $\nrq$ and $\nrp$ are odd, one can express $q=r_1+s_1$ and $p=r_2+s_2$, where $r_1,r_2\in 2\J$, 
					and $s_1,s_2\in\set{\e,\ii,\j,\k}$.  With this, one obtains
					\[qw\overline{p}=r_1w\overline{r_2}+r_1w\overline{s_2}+s_1w\overline{r_2}+s_1w\overline{s_2}\in\L.\]
					By Fact~\ref{f1}, $r_1w\overline{r_2}+r_1w\overline{s_2}+s_1w\overline{r_2}\in 2\J\subseteq\L$, which implies that $s_1w\overline{s}_2\in\L$.  Thus,
					$v=Rw\in R\G_2$ and $\G_2\cap\G_1(R)\subseteq R\G_2\cap\G_1(R)$.  It follows then that $R\G_2\cap\G_1(R)=\G_2\cap\G_1(R)$, and so $R\in H_{\G_1,\G_2}$. 
					
				\noindent\textsc{Case II:} $\nrq$ is odd and $\nrp\equiv 0\imod{4}$, or $\nrq\equiv 0\imod{4}$ and $\nrp$ is odd

					Consider $x=\frac{1}{2}\abs{q\,\overline{p}}\in\L$.  One has $R(x)=qx\overline{p}/\abs{q\,\overline{p}}=\frac{1}{2}q\,\overline{p}\in\J$.  Thus, 
					$R(x)\in R\G_2\cap\G_1(R)$.  However, the first component of $q\,\overline{p}$, which is equal to the inner product $\inn{q}{p}$ of $q$ and $p$
					(or the standard scalar product of $q$ and $p$ as vectors in $\R^4$), is odd. This implies that
					$\frac{1}{2}q\,\overline{p}\notin\L$ and $R(x)\notin\G_2\cap\G_1(R)$.  Hence, $R\G_2\cap\G_1(R)\neq \G_2\cap\G_1(R)$ and $R\notin H_{\G_1,\G_2}$.

				\noindent\textsc{Case III:} $\nrq\equiv\nrp\equiv 2\imod{4}$
					
					Write $q=r_1+s_1$ and $p=r_2+s_2$ where $r_1$, $r_2\in 2\J$ and \[s_1, s_2\in\set{(1,1,0,0),(1,0,1,0),(1,0,0,1)}.\]  Note that 
					$\inn{q}{p}$ is	even if and only if $s_1=s_2$.

					Consider again $x=\frac{1}{2}\abs{q\,\overline{p}}\in\L$ so that $R(x)=\frac{1}{2}q\,\overline{p}\in R\G_2\cap\G_1(R)$.  Now, 
					$\frac{1}{2}q\,\overline{p}\notin\L$ if and only if $\inn{q}{p}$ is odd.  Since $\frac{1}{2}q\,\overline{p}\notin\L$ implies that 
					$R\G_2\cap\G_1(R)\neq\G_2\cap\G_1(R)$, $R\notin H_{\G_1,\G_2}$ whenever $\inn{q}{p}$ is odd.

					It remains to check the case $s_1=s_2$.  Take $v\in\G_2\cap\G_1(R)$.  Write $v=qw\overline{p}/\abs{q\,\overline{p}}$ for some $w\in\J$.  One has
					\[\tfrac{1}{2}\abs{q\,\overline{p}}v=\tfrac{1}{2}(r_1w\overline{r_2}+r_1w\overline{s_1}+s_1w\overline{r_2}+s_1w\overline{s_1})\in\L.\]
					Now, $\frac{1}{2}r_1w\overline{r_2}\in 2\J\subseteq\L$ and $\tfrac{1}{2}r_1w\overline{s_1}$, $\tfrac{1}{2}s_1w\overline{r_2}\in 
					(1,1,0,0)\J\subseteq\L$ by Fact~\ref{f1}.  Hence, $\tfrac{1}{2}s_1w\overline{s_1}=s_1w{s_1}^{-1}\in\L$ 
					meaning $w\in{s_1}^{-1}\L s_1=\L$ for all three possible values of $s_1$.  It follows then that $v\in R\G_2$ and $R\G_2\cap\G_1(R)=\G_2\cap\G_1(R)$.  
					Thus, $R\in H_{\G_1,\G_2}$ if $\inn{q}{p}$ is even.
					
				\noindent\textsc{Case IV:} $\nrq\equiv\nrp\equiv 0\imod{4}$
	 
					Write $q=r_1+s_1$ and $p=r_2+s_2$ where $r_1,r_2\in (1,1,0,0)2\J$ and \[s_1,s_2\in\set{(1,1,1,1),(1,1,1,-1)}.\]  Note that $4\mid\inn{q}{p}$ 
					if and only if $s_1=s_2$. 
	 
					Take $x=\frac{1}{4}\abs{q\,\overline{p}}\in\L$.  One obtains $R(x)=\lt(\frac{1}{2}q\rt)\lt(\frac{1}{2}\overline{p}\rt)\in\J$. Hence, $Rx\in
					R\G_2\cap\G_1(R)$.  Observe however that $\frac{1}{4}q\,\overline{p}\notin\L$ if and only if $4\nmid\inn{q}{p}$.  Since 
					$\frac{1}{4}q\,\overline{p}\notin\L$ means that $R\G_2\cap\G_1(R)\neq\G_2\cap\G_1(R)$, $4\nmid\inn{q}{p}$ implies $R\notin H_{\G_1,\G_2}$.
	 
					Again, it remains to check the instance when $s_1=s_2$.  Let $v\in\G_2\cap\G_1(R)$.  If one writes $v=qw\overline{p}/\abs{q\,\overline{p}}$ for 
					some $w\in\J$, one gets 
					\[\tfrac{1}{4}\abs{q\,\overline{p}}v=\tfrac{1}{4}(r_1w\overline{r_2}+r_1w\overline{s_1}+s_1w\overline{r_2}+s_1w\overline{s_1})\in\L.\]
					By Fact~\ref{f1}, $\frac{1}{4}r_1w\overline{r_2}$, $\tfrac{1}{4}r_1w\overline{s_1}$, and 
					$\tfrac{1}{4}s_1w\overline{r_2}\in (1,1,0,0)\J\subseteq\L$.  Therefore, $w\in\L$ for both possible values of $s_1$.  Hence, $v\in R\G_2$ which implies 
					that $R\G_2\cap\G_1(R)=\G_2\cap\G_1(R)$.  Therefore, $R\in H_{\G_1,\G_2}$ whenever $\inn{q}{p}$ is divisible by 4.
					
				\smallskip The results above also hold for coincidence reflections $T_{q,p}=R_{q,p}\cdot T_{1,1}\in \OC(\G_1)$.		
	\bibliographystyle{amsplain}
	\bibliography{bibliog}
\end{document}